\documentclass[a4paper,reqno]{amsart}
\usepackage{amssymb}
\usepackage{latexsym}
\usepackage{amsmath}
\usepackage{euscript}
\usepackage{graphics,color}
\usepackage[all,cmtip]{xy}
\usepackage[margin=3cm]{geometry}
\usepackage{MnSymbol} % for \udots
\usepackage{amsmath}
\usepackage{tcolorbox}
\usepackage{mathrsfs}
\usepackage{pgf,tikz,pgfplots}
\pgfplotsset{compat=1.15}
\usetikzlibrary{arrows}
\usepackage{graphicx}

\DeclareMathOperator{\diag}{diag}

\newcommand{\rFs}[5]{\,_{#1}F_{#2} \left( \genfrac{.}{.}{0pt}{}{#3}{#4};#5 \right)}

      \def\dC{{\mathbb C}}

   \def\dN{{\mathbb N}}

\newcommand{\be}{\begin{equation}}
\newcommand{\ee}{\end{equation}}
\newcommand{\ba}{\begin{eqnarray}}
\newcommand{\ea}{\end{eqnarray}}
\newcommand{\baa}{\begin{eqnarray*}}
\newcommand{\eaa}{\end{eqnarray*}}
\newcommand{\bb}{\begin{thebibliography}}
\newcommand{\eb}{\end{thebibliography}}

\newcounter{my}
\newcommand{\he}%
   {\stepcounter{equation}\setcounter{my}%
   {\value{equation}}\setcounter{equation}0%
   }%
\newcommand{\she}%
   {\setcounter{equation}{\value{my}}%
    }%

\DeclareMathOperator\End{ End}
\DeclareMathOperator\rank{ rank}

\newtheorem{theorem}{Theorem}[section]
\newtheorem{proposition}[theorem]{Proposition}
\newtheorem{corollary}[theorem]{Corollary}
\newtheorem{lemma}[theorem]{Lemma}
\newtheorem{definition}[theorem]{Definition}
\theoremstyle{definition}

\newtheorem{remark}[theorem]{Remark}

\numberwithin{equation}{section}

\newcommand{\bra}[1]{\langle\,{#1}}
\newcommand{\ket}[1]{\mid{#1}\,\rangle}

\title[Algebraic Interpretation of discrete MVOPs]{Algebraic Interpretation of discrete families of matrix valued orthogonal polynomials}
\author{Quentin Labriet, Lucia Morey and Luc Vinet}
\address{Centre de recherches math\'ematiques, Universit\'e de Montr\'eal, P.O. Box 6128, Centre-ville Station, Montr\'eal (Qu\'ebec), H3C 3J7}

\address{Centre de recherches math\'ematiques, Universit\'e de Montr\'eal, P.O. Box 6128, Centre-ville Station, Montr\'eal (Qu\'ebec), H3C 3J7}

\address{IVADO and Centre de recherches math\'ematiques, Universit\'e de Montr\'eal, P.O. Box 6128, Centre-ville Station, Montr\'eal (Qu\'ebec), H3C 3J7}

\begin{document}

\begin{abstract}
An algebraic interpretation of matrix valued orthogonal polynomials (MVOPs) is provided. The construction is based on representations of a ($q$-deformed) Lie algebra $\mathfrak{g}$ into the algebra $\End_{M_n(\dC)}(M)$ of $M_n(\dC)$-linear maps over a $M_n(\dC)$-module $M$. Cases corresponding to the Lie algebras $\mathfrak{su}(2)$ and $\mathfrak{su}(1, 1)$ as well as to the $q$-deformed algebra $\mathfrak{so}_q(3)$ at $q$ a root of unity are presented; they lead to matrix analogs of the Krawtchouk, Meixner and discrete Chebyshev polynomials.
\end{abstract}

\maketitle

\section{Introduction}
The theory of matrix valued orthogonal polynomials (MVOPs for short) was introduced by M.G. Krein in 1949 \cite{Krein1}. Since then, numerous theoretical advancements have been made, closely following the development of scalar valued orthogonal polynomials. A comprehensive overview of these developments can be found in \cite{DamanikPushnitskiSimon08}. In a manner similar to scalar orthogonal polynomials, there is considerable interest in studying families of MVOPs that also exhibit the additional property of being eigenfunctions of second-order differential operators \cite{DurandlI2008_2,DuranG1,GrunbaumPacharoniTirao02,KoelinkRiosRoman17}, difference operators \cite{ALVAREZNODARSE201340,DURAN2012586,EMR}, or $q$-difference operators \cite{ALDENHOVEN2015164,AKR1}.  One method for constructing such examples is by utilizing the representation theory of compact Lie groups and harmonic analysis on compact symmetric spaces. In this context, scalar valued polynomials are known to arise as matrix coefficients or spherical functions. The spherical function approach has been extended to matrix analogues and successfully applied in works such as \cite{GrunbaumPacharoniTirao02,HeckmanVP,KvPR1,KvPR2} among  others. Despite these advancements, there are still relatively few concrete examples of MVOPs, particularly those that satisfy a difference equation. More specifically, our goal is to identify examples of MVOPs along with their recurrence relations, orthogonality relations, and, when applicable, their corresponding differential or difference equations. 

The goal of this article is to provide an algebraic construction for MVOPs as transition coefficients between two specified bases. We adopt the point of view introduced in \cite{GranovskiiZhedanov86} and much developed by one of us often in collaboration with Zhedanov, see in particular the reviews \cite{Genest2014racah}, \cite{De2015bannai}.  Specifically, an approach for providing an algebraic interpretation of the polynomial families in the Askey scheme involves the algebra generated by two operators: the multiplication by the operator multiplication by the variable and the 
$(q)$-difference or differential operator. Depending on the case, the resulting algebra can either be a  Lie algebra (or its $q$ deformation), or a more general quadratic algebra (such as Hahn or Racah) which are specializations of the Askey-Wilson algebra.

In \cite{Floreanini1993quantum,GranovskiiZhedanov86},  scalar valued orthogonal polynomials are obtained as transition coefficients between two eigenbases of two self-adjoint operators $H$ and $P$ inside a Lie algebra. If $H$ admits a discrete spectrum and $P$ acts tridiagonally on the eigenbasis of $H$ then the transition coefficients between both eigenbases can be seen to be orthogonal polynomials with the eigenvalues of $P$ as variables.  This method allows for the uncovering of orthogonality relations, three term recurrence relations, and differential or difference equations satisfied by the polynomials. This known interpretation for scalar orthogonal polynomials, represents a new and intriguing research direction for MVOPs.

To apply such a construction, a natural framework is  $M_n(\dC)$ modules that admits a matrix valued inner product. More precisely, we look at representations of a Lie (and $q$-deformed) algebra $\mathfrak{g}$ into the algebra $\End_{M_n(\dC)}(M)$ of $M_n(\dC)$-linear maps over a left $M_n(\dC)$-module $M$, i.e. Lie algebra morphisms between $\mathfrak{g}$ and $\End_{M_n(\dC)}(M)$. Similarly as in the scalar case, we have two self-adjoint operators $H$ and $P$ with $P$ acting tridiagonally on the eigenbasis of $H$, then MVOPs will be obtained as transition coefficients between both eigenbases. 

In the recent preprint \cite{koelinkromanzudilim24}, the authors emphasized that naturally defined MVOP are close to their scalar valued prototypes. In some sense the present work follows the same philosophy since families of MVOP are obtained naturally from the algebra related to families of scalar valued polynomials and their representations. More precisely, since $M_n(\dC)$ is Morita equivalent to $\dC$ there is an equivalence of categories between $M_n(\dC)$ modules and complex vector spaces. Thus starting from a complex representation of our algebra $\mathfrak{g}$ we get a representation  in $\End_{M_n(\dC)}(V)$.
Using this equivalence, we can unfold some properties of MVOPs starting from the scalar valued case. 

In Section $2$ we present the necessary background for the construction, i.e. Hermitian modules and Morita equivalence.  Section $3$ introduces the general construction for MVOPs seen as transition coefficients. In the last part of the article, we use this algebraic presentation to construct some families of discrete MVOPs together with their orthogonality, recurrence relation and difference equation. This goal is achieved in Section $4$ to $6$, using three different algebras: $\mathfrak{su}(2)$, $\mathfrak{su}(1,1)$ and $\mathfrak{so}_q(3)$. This leads to the construction of, respectively, Krawtchouk, Meixner and Discrete Chebyshev matrix analogues.

\subsection{Notations} In all the paper we use $A^*$ to denote the transpose conjugate of a square matrix $A$. For a symmetric matrix $ A$, we write $A \geq 0$ for non negative definite matrices and $A>0$ for positive definite ones.

%%%%%%%%%%%%%%%%%%%%%%%%%%%%%%%%%
%%%%%%%%%%%%%%%%%%%%%%%%%%%%%%%%%
%%%%%%%%%%%%%%%%%%%%%%%%%%%%%%%%%
%%%%%%%%%%%%%%%%%%%%%%%%%%%%%%%%%

\section{$M_n(\dC)$-module}
\subsection{Morita equivalence}
In this subsection we discuss $M_n(\dC)$-modules (thought as left modules) with matrix valued inner products, that will be called Hermitian modules. First, it is a fact  that the algebras $M_n(\dC)$ and $\dC$ are Morita equivalent, see \cite{Lam99}, i.e. their respective module categories $\mathfrak{M}_{M_n(\dC)}$ and $\mathfrak{M}_{\mathbb{C}}$  are equivalent. The functors realizing this equivalence are given by 
\begin{equation}
F: \mathfrak{M}_{M_n(\mathbb{C})}\to \mathfrak{M}_{\mathbb{C}}, \qquad G: \mathfrak{M}_{\mathbb{C}}\to \mathfrak{M}_{M_n(\mathbb{C})},    
\end{equation}
on the objects, the functors are defined by 
\begin{equation}
    G(V)=V^n, \qquad F(M)=e_{11}M ,
\end{equation}
where $e_{ij}$ denote the elementary matrices. Here $V^n$ is a $M_n(\dC)$-module, the action being given, for $A\in M_n(\dC)$, by:
\begin{equation}\label{eq:MnActionOnVn}
A (v_1,\ldots, v_n) =  \left(\sum_k a_{1k}v_k, \ldots, \sum_n a_{nk}v_k \right). 
\end{equation}
On the morphisms, the functors are defined by 
 \begin{align}
 \alpha&:V\to V', \quad G(\alpha): G(V)\to G(V'),\quad G(\alpha)\cdot(v_1,\ldots,v_n)=(\alpha v_1, \ldots, \alpha v_n),\\
\beta&:M\to M', \quad F(\beta): F(M)\to F(M'),\quad F(\beta)\cdot(e_{11} m)= e_{11}\beta(m).
 \end{align}

We now introduce introduce Hermitian $M_n(\dC)$-modules. This definition is adapted to our purpose from the one in \cite[Dfn II.7.1.1]{Blackadar} for Hilbert modules.  
\begin{definition}
A Hermitian $M_n(\dC)$-module $M$ is a $M_n(\dC)$-module together with a matrix valued inner product, i.e. a map $(\cdot,\cdot): M\times M\to M_n(\dC)$ such that:
\begin{enumerate}
    \item For all $A\in M_n(\dC)$, $ (A m_1,m_2)=A(m_1,m_2)$,
    \item $(m_1,m_2)=(m_2,m_1)^*$,
    \item $(m,m)\geq 0$,
    \item If $(m,m)=0$ then $ m=0$ (we say that the inner product is non degenerate).
\end{enumerate} 
\end{definition}

In general the non degeneracy condition (4) can be omitted but in our case every inner product will satisfy it.  Another remark is that as a direct consequence of the definition we get
\begin{equation}
(m_1,A m_2)=(m_1,m_2)A^*.    
\end{equation}

We will now explicitly extend the Morita equivalence from Hermitian spaces $V$ to Hermitian $M_n(\dC)$-modules $V^n$. Let $(V,\langle \cdot,\cdot \rangle)$ be a Hermitian space and endow the $M_n(\dC)$-module $G(V)=V^n$ with the matrix valued inner product $G(\langle\cdot,\cdot\rangle)=(\cdot,\cdot)_G$ defined by 
\begin{equation}
\label{eq:mv-innerproduct}
    \left( ( v_1,\ldots, v_n),(w_1,\ldots, w_n)\right)_G = (\langle v_i,w_j\rangle) \in M_n(\dC). 
\end{equation}
The first two properties of matrix valued inner products are routine to check. For the third one, we observe that the matrix $(v,v)$ is the Gram matrix of the family $(v_1,\ldots,v_n)$ which is known to be non-negative. The inner product $(\cdot,\cdot)_G$ is non degenerate which is easily seen by looking at the diagonal coefficients.  

In the other direction, starting with a Hermitian $M_n(\dC)$-module $(V^n,(\cdot,\cdot))$, we endow $F\circ V^n$ with the inner product $F((\cdot,\cdot))=\langle \cdot, \cdot\rangle_F$ defined by 
\begin{equation}
    \langle v,w\rangle_F =( (v,0,\ldots,0),(w,0,\ldots,0))_{1,1}.
\end{equation}
We have the following
\begin{equation}
((v,0,\ldots,0),(w,0,\ldots,0))=(e_{11}(v,0,\ldots,0),e_{11}(w,0,\ldots,0))=e_{11}((v,0,\ldots,0),(w,0,\ldots,0))e_{11},
\end{equation}
so that the only non zero coefficient in $((v,0,\ldots,0),(w,0,\ldots,0))$ is the upper left one. This implies that $\langle\cdot,\cdot\rangle_F$ is indeed a non degenerate inner product on $V$. 
We have 
\[
G\circ F( (\cdot,\cdot))=(\cdot,\cdot),\quad F\circ G (\langle \cdot, \cdot\rangle)=\langle \cdot, \cdot\rangle. 
\]
The second equation is direct from the definitions. Regarding the first one, introduce $\sigma_{1j}$ to be the permutation matrix associated with the transposition $(1,j)$ we have
\begin{align}
    ((v_1,\ldots,v_n),(w_1,\ldots,w_n))_{ij}&=\sum_{k,\ell} \left(\sigma_{1k} ((v_k,0,\ldots,0),(w_\ell,0,\ldots,0))\sigma_{1\ell}\right)_{ij}\nonumber\\
    &=((v_i,0,\ldots,0),(w_j,0,\ldots,0))_{11}.
\end{align}
It is then showed to be equal to $G\circ F( (\cdot,\cdot))((v_1,\ldots,v_n),(w_1,\ldots,w_n))_{ij}$. 

As a summary of this discussion we showed a ``unitary'' Morita equivalence (the morphisms of the module categories being the same)
\begin{equation}
G :(V,\langle\cdot,\cdot\rangle) \mapsto (V^n,(\cdot,\cdot)_G),\quad F :(V^n,(\cdot,\cdot))\mapsto (V,\langle\cdot,\cdot\rangle_F).
\end{equation}
To conclude this discussion on the Morita equivalence we define the adjunction on $\End_{M_n(\dC)}(V^n)$ in a natural way .

\begin{definition}
For an operator $T\in \End_{M_n(\dC)}(V^n)$, we define the adjoint of $T$ with respect to  $(\cdot,\cdot)$ by the property
\begin{equation}
    (T\cdot m_1,m_2)=(m_1,T^\dagger \cdot m_2).
\end{equation}
Moreover, we say that $T$ is self-adjoint if $T^\dagger=T$. 
\end{definition}

We will abuse notations and use also ${}^\dagger$ for the usual adjoint on Hermitian spaces. The existence of the adjoint is proved through the operator $G\circ (F\circ T)^{\dagger}$, and the uniqueness is then a consequence of the Morita equivalence. Using the definitions one proves the following lemma. 
\begin{lemma}
    The functors $G$ and $F$ preserve adjoints, i.e. 
    \begin{equation}
    (G\circ T)^\dagger= G\circ T^\dagger, \qquad (F\circ T)^\dagger= F\circ T^\dagger. 
    \end{equation}
\end{lemma}

\begin{proof}
   This is based on the following for $v=(v_1,\ldots,v_n)$ and $w=(w_1,\ldots, w_n)$
    \begin{equation}
    (G\circ T \cdot v,w)_{i,j}=(\langle T\cdot v_i,w_j\rangle)=(\langle v_i, T^\dagger \cdot w_j\rangle)=(v,G\circ T^\dagger \cdot w)_{i,j}. 
    \end{equation}
    This concludes the proof of the lemma. 
\end{proof}

As a consequence, self-adjoint operators are send to self-adjoint operators. The same is true for unitary operators, defined in a natural way for Hermitian $M_n(\dC)$-modules. 

%%%%%%%%%%%%%%%%%%%%%%%%%%%%%%%%%
%%%%%%%%%%%%%%%%%%%%%%%%%%%%%%%%%
%%%%%%%%%%%%%%%%%%%%%%%%%%%%%%%%%
%%%%%%%%%%%%%%%%%%%%%%%%%%%%%%%%%
\subsection{Free modules}

In Section 3 our focus will be on free modules  $V^n$. Free modules admit bases that will be called $M_n(\dC)$-bases to avoid confusion with vector spaces bases on $V^n$. For the ring $M_n(\dC)$ all bases have the same cardinal called the rank of $V^n$. 

\begin{lemma}\label{lem:FreeModuleBasis}
    If $V^n$ is a free module of $M_n(\dC)$ and $e_i=(e_1^i,\ldots, e_n^i)$ is a $M_n(\dC)$-basis of $V^n$ then the family $(e_i^j)_{i,j}$ is a basis of $V$. 
\end{lemma}
\begin{proof}
    This is a direct consequence of equation \eqref{eq:MnActionOnVn}. Indeed, any element in $V^n$ can be written in a unique way as
    \begin{equation}
    \sum_i A_i e_i =\left( \sum_{i,k}a^i_{1,k}e^i_k,\ldots, \sum_{i,k}a^i_{n,k}e^i_k\right).
    \end{equation}
    Thus any $v\in V$ can be written in a unique way as a linear combination of the $e_k^i$. This concludes the proof of the lemma.
\end{proof}
The next corollary gives a characterization of free module of finite rank. 
\begin{corollary}
\label{cor:free-ndivdimV}
A $M_n(\dC)$-module $V^n$ is a free module of finite rank if and only if $V$ is finite dimensional and $n$ divides $\dim V$. 
\end{corollary}

\begin{proof}
    Assume $V^n$ is a free $M_n(\mathbb{C})$ module of finite rank. The previous Lemma proves that $n$ divides $\dim V$. More precisely, $n \rank (M) =\dim V$. Conversely, if $n$ divides $\dim V$. Let $|k\rangle$ be a basis of $V$. Then the family 
    \begin{equation}\label{eq:ModuleGenerators}
    e_i = (|ni\rangle, \ldots, |ni+n-1\rangle),
    \end{equation}
is a $M_n(\dC)$-basis of $V^n$ as seen using equation \eqref{eq:MnActionOnVn}. 
\end{proof}
When $V^n$ is a module of finite rank, we will write is rank $N+1$ so that $\dim (V)=n(N+1)$.

\begin{remark}
    Notice that if $V^n$ is not of finite rank but of countable rank then the family \eqref{eq:ModuleGenerators} is still a $M_n(\dC)$-basis of $V^n$. Thus for free module of countable rank there is a correspondence between $M_n(\dC)$-basis of $V^n$ and basis of $V$. However, it is not one-to-one since two bases of $V$ can lead to two different $M_n(\dC)$-bases of $V^n$ for example by permuting the components of the generators $e_i$. 
\end{remark}

We now assume that the $M_n(\dC)$-modules are all endowed with a matrix valued inner product. We are then naturally interested in orthogonal $M_n(\dC)$-basis $(e_k)_k$ which satisfies $(e_k,e_\ell)=0$ for $k\neq \ell$, and in orthonormal basis which moreover satisfies $(e_k,e_k)=I_n$. It is not true that an orthogonal $M_n(\dC)$-basis $e_i$ of $V^n$ will lead to an orthogonal basis for $V$. It is true if and only if the inner product $(e_k,e_k)$ is a diagonal matrix. However, orthonormal bases are preserved via Lemma \ref{lem:FreeModuleBasis}. 

\begin{corollary}
    If $V^n$ is a free Hermitian module of countable rank then  $e_i=(e_1^i,\ldots, e_n^i)$ is an orthonormal $M_n(\dC)$-basis of $V^n$ if and only if the family $(e_i^j)_{i,j}$ is an orthonormal basis of $V$.
\end{corollary}

A last point in this section is to relate eigenbases of operators on $V$ to $M_n(\dC)$-eigenbases of morphisms on $V^n$. The following definition makes explicit classical notions in the context of $M_n(\dC)$-modules. 

\begin{definition}
    Let $V^n$ be a free $M_n(\dC)$-module and $T\in \End_{M_n(\dC)}(V^n)$. A family $(e_i)_i$ is a $M_n(\dC)$-eigenbasis of $T$ if 
    \begin{equation}
    T\cdot e_i=\Lambda_i e_i,
    \end{equation}
    with $\Lambda_i\in M_n(\dC)$, and $(e_i)_i$ is an $M_n(\dC)$-basis of $V^n$. We also say $T$ is diagonalisable when $T$ admits an eigenbasis. 
\end{definition}

The next lemma relates diagonalisability on $V$ and on $V^n$. 
\begin{lemma}\label{lem:Diagonalisability}
    Let $V^n$ be a free module of countable rank. $H\in \End_\dC(V)$ is diagonalisable if and only if $G\circ H \in \End_{M_n(\dC)}(V^n)$ is diagonalisable with diagonalisable eigenvalues $\Lambda_i$.
\end{lemma}
\begin{proof}
On one hand, if $H$ is diagonalisable with eigenvectors $\ket{i}$ then the basis $e_i=(\ket{ni},\ldots,\ket{ni+n-1})$ is a diagonal eigenbasis for $G\circ H$. Moreover the eigenvalues $\Lambda_i$ associated to $e_i$ are diagonal matrices. 

On the other hand, if $G\circ H$ is diagonalisable with $M_n(\dC)$-eigenbasis and $\Lambda_i$ is diagonalized by $P_i^{-1} D_i P_i$ then
    \begin{equation}
    G\circ H (P_i e_i)= D_i \left(P_ie_i\right). 
    \end{equation}
    Since $P_i$ is invertible the family $(P_i e_i)_i $ is a $M_n(\dC)$-eigenbasis of $G\circ H$ with diagonal eigenvalue $D_i$. The coordinates of this $M_n(\dC)$-basis form a basis of $V$ which is an eigenbasis for $H$. 
\end{proof}

The following corollary is a generalization of the spectral theorem for self-adjoint operators. 
\begin{corollary}\label{coro:SelfAdjointDiagonalisability}
Let $V^n$ be a free Hermitian module of countable rank. Every self-adjoint operator for $V^n$ admits an orthonormal $M_n(\dC)$-eigenbasis. 
\end{corollary}
\begin{proof}
   If $T$ is a self-adjoint operator then so is $F\circ T$, thus it admits an eigenbasis $\ket{k}$ of $V$ made of orthonormal generators.
   Then the family $e_i=(\ket{ni},\ldots,\ket{ni+n-1})$ is an orthonormal $M_n(\dC)$-eigenbasis. 
\end{proof}

%%%%%%%%%%%%%%%%%%%%%%%%%%%%%%%%%
%%%%%%%%%%%%%%%%%%%%%%%%%%%%%%%%%
%%%%%%%%%%%%%%%%%%%%%%%%%%%%%%%%%
%%%%%%%%%%%%%%%%%%%%%%%%%%%%%%%%%
\section{Algebraic interpretation for MVOPs}
\label{sec:algebraic-interpretation-mvop}
In this section we explain how one can recover MVOPs as transition coefficients for representations of algebras over Hermitian $M_n(\dC)$-modules. Let $\mathfrak{g}$ be a ($q$-deformed) Lie algebra, and $(V,\rho)$ be a unitary representation of $\mathfrak{g}$. Then we can create a representation of $\mathfrak{g}$ on $\End_{M_n(\dC)}(V^n)$ using the following commutative diagram.

\[
  \xymatrix{  \End_{M_n(\dC)}(V^n) \ar@<2pt>[rr]^F  && \End_\dC(V) \ar@<2pt>[ll]^G \\ & \ar[lu]^{\tau} \mathfrak{g} \ar[ru]_\rho }
\]
From now on we will assume that $V$ has finite dimension $\dim(V)=n(N+1)$.
Let $H$ be an element of $\mathfrak{g}$ which admits an orthonormal $M_n(\dC)$-eigenbasis $(e_k)_k$ with eigenvalues $\Lambda_k\in M_n(\dC)$:
\begin{equation}
H\cdot e_k=\Lambda_k e_k,\qquad k=0,\ldots,N,
\end{equation}
and $P$ an operator, which acts tridiagonally on the basis $(e_k)_k$ and is self-adjoint. This operator can be described as in the following proposition.

\begin{proposition}
\label{pr:3TR-selfadjoint}
Let $(e_k)_k$ be a orthogonal $M_n(\dC)$-basis of $V^n$ and $P\in \End_{M_n(\dC)}(V^n)$ such that
\begin{equation}\label{eq:action-self-adjoint-operator}
    P\cdot e_k = A_{k+1}^\ast e_{k+1} +B_k e_k + A_k e_{k-1},
\end{equation}
where $B_k=B_{k}^\ast.$ Then $P$ is self-adjoint with respect to the matrix valued inner product \eqref{eq:mv-innerproduct}. 
Conversely, if $P$ is self-adjoint and acts tridiagonally on the basis $(e_k)_k$ then $P$ is described by formula \eqref{eq:action-self-adjoint-operator}.
\end{proposition}

\begin{proof}
The proof follows verifying that $(P\cdot e_k, e_\ell)=(e_k, P \cdot e_\ell)$ for all $k, \ell.$ On one hand, we have
$$(P\cdot e_k, e_\ell)=(A_{k+1}^\ast e_{k+1} +B_k e_k + A_k e_{k-1},e_\ell)=A_{k+1}^\ast (e_{k+1},e_\ell) + B_k (e_k,e_\ell) +A_k(e_{k-1},e_\ell).$$
On the other hand, we have
$$( e_k, P\cdot e_\ell)=(e_k, A_{\ell+1}^\ast e_{\ell+1}+B_\ell e_\ell +A_\ell e_{\ell-1})= (e_{k},e_{\ell+1})A_{\ell+1} + B_k^\ast (e_k,e_\ell) +(e_{k},e_{\ell-1})A_{\ell}^\ast$$
Since $(e_k)_k$ is an orthogonal $M_n(\dC)$-basis of $V^n$, and $B_k=B_k^\ast$ we get the desired result. 

The converse is obtained starting from a tridiagonal action $P\cdot e_k=\alpha_k e_{k+1}+ \beta_k e_k +\gamma_k e_{k-1}$ and using that $P$ is self-adjoint. 
\end{proof}

According to Corollary \ref{coro:SelfAdjointDiagonalisability}, the operator $P$ admits an orthonormal $M_n(\dC)$-eigenbasis $(\phi_j)_j$ with eigenvalues $\alpha_j\in M_n(\dC)$:
\begin{equation}
    P\cdot \phi_j = \alpha_j\phi_j\,,\quad j=0,\ldots,N.
\end{equation}

Now consider the transition coefficients $( e_k,\phi_j)$, and define the family $\Pi_k(\theta_j)$ where $\theta_j=( e_0, \phi_j) {\alpha}_j^*( e_0, \phi_j)^{-1}$ by
\begin{equation}
\label{eq:transition-coef-general}
( e_k,\phi_j) =\Pi_k(\theta_j)( e_0, \phi_j),\quad k,j=0,\ldots,N,
\end{equation}
which is possible as soon as the matrix $( e_0, \phi_j)$ is invertible. 
We will assume this is the case in the rest of the section. The family $\Pi_k(\theta_j)$ satisfies the following three term recurrence relation:
\begin{equation}
\label{eq:recurrence-pi}
    A_{k+1}^*\Pi_{k+1}(\theta_j)+B_k\Pi_k(\theta_j)+A_k\Pi_{k-1}(\theta_j)=\Pi_k(\theta_j)\theta_j, \qquad \Pi_{-1}(\theta_j)=0,\qquad \Pi_0(\theta_j)=I_n.
\end{equation}

We will now see that this three term recurrence relation uniquely determines a family of matrix valued polynomials in one variable $\Pi_k(x)$.

\begin{definition}
\label{def:vandermonde}
Let $\theta_j\in M_n(\dC)$ be an entrywise rational function of $j$
without poles in $0,\ldots,N$. We say that $(\theta_j)_j$ satisfies the Vandermonde condition if for every $k\in 1,\ldots,N$, the following block Vandermonde matrix is invertible
\begin{equation}
\label{eq:vandermonde}
V(\theta_0,\ldots,\theta_k)=\begin{pmatrix} I & I & \cdots & I \\
\theta_0 & \theta_1& \cdots & \theta_k \\
\vdots & \vdots & \ddots & \vdots \\
\theta_0^k &\theta_1^k & \cdots & \theta_k^k
\end{pmatrix}.
\end{equation}
\end{definition}
Notice that the Vandermonde condition being satisfied implies $\theta_i\neq \theta_j$ for all $i\neq j$. However, the converse is not true in general.  
\begin{proposition}
\label{thm:interpolation}
Given $k$ pairs of matrices $(\theta_i,Y_i)$ $i=0,\ldots,k-1$, such that the Vandermonde matrix $V(\theta_0,\ldots,\theta_{k-1})$ is invertible, there exists a unique monic matrix valued polynomial $P(x)=x^{k}+M_{k-1} x^{k-1}+\ldots+M_0\,,$ $x\in M_n(\mathbb{C})$
which satisfies $P(\theta_i)=Y_i$ for $i=0,\ldots,k-1$.
\end{proposition}
\begin{proof}
The equations $P(\theta_i)=Y_i$ for $i=0,\ldots,k$ can be reformulated as
    \begin{equation}
  \begin{pmatrix}
    M_0,M_1,\ldots,M_{k-1}
\end{pmatrix}\begin{pmatrix} I & I & \cdots & I \\
\theta_0 & \theta_1& \cdots & \theta_{n-1} \\
\vdots & \vdots & \ddots & \vdots \\
\theta_0^{k-1} &\theta_1^{k-1} & \cdots & \theta_{k-1}^{k-1}
\end{pmatrix}=\begin{pmatrix}
    Y_0-\theta_0^k,Y_1-\theta_0^k,\ldots,Y_{k}-\theta_{k-1}^k
\end{pmatrix}. 
\end{equation} 
Since the Vandermonde matrix $V(\theta_0, \ldots, \theta_k)$ is invertible, the coefficients $M_0, \ldots, M_{k-1}$ are uniquely defined and the matrix polynomial $P$ satisfying $P(\theta_i)=Y_i$ is unique. 
\end{proof}
\begin{proposition}
If $(\theta_j)_j$ satisfies the Vandermonde condition and $A_k$ (defined in \eqref{eq:action-self-adjoint-operator}) are invertible for all $k$ then there exists a unique family of matrix valued polynomials 
\begin{equation}
    \Pi_k(x)=M_{k,k}x^k+M_{k,k-1}x^{k-1}+\ldots+M_{k,0},\quad x\in M_n(\mathbb{C}),
\end{equation}
with invertible leading coefficient satisfying \eqref{eq:recurrence-pi}. Moreover, the polynomials $\Pi_k(x)$ satisfy the orthogonality relation:
\begin{equation}
    \sum_{j=0}^{N+1} \Pi_k(\theta_j) W(j) \Pi_l(\theta_j)^*=\delta_{k,l},
\end{equation}
where $W(j)=( e_0, \phi_j)( e_0, \phi_j)^*$.
\end{proposition}
\begin{proof}
Let $\widetilde\Pi_k(\theta_j)=A_{k}^*\ldots A_1^* \Pi_k(\theta_j)$, equation \eqref{eq:recurrence-pi} can be rewritten in monic form as
\begin{equation}
\label{eq:recurrence-monicpi}
\widetilde\Pi_{k}(\theta_j)\theta_j=\widetilde\Pi_{k+1}(\theta_j)+B_k \widetilde\Pi_k(\theta_j)+A_kA_k^* \widetilde\Pi_k(\theta_j), \qquad\widetilde\Pi_{-1}(\theta_j)=0,\qquad \widetilde\Pi_0(\theta_j)=I_n\,.
\end{equation}
The proof follows by applying Proposition \ref{thm:interpolation} to the pairs of matrices $(\theta_j,Y_{k+1,j})$, $j=0,\ldots,k$, with 
$Y_{k+1,j}=\widetilde\Pi_{k}(\theta_j)\theta_j-B_k \widetilde\Pi_k(\theta_j)-A_kA_k^* \widetilde\Pi_k(\theta_j).$ The orthogonality relation is a consequence of the orthogonality of the basis $e_k$ and $\phi_j$.
\begin{equation}
    \delta_{k,l}=( e_k,e_l) =\sum_{j=0}^{N+1} ( e_k,\phi_j) ( e_l,\phi_j)^*=\sum_{j=0}^{N+1} \Pi_k(\theta_j) W(j) \Pi_l(\theta_j)^*,
\end{equation}
This concludes the proof of the proposition.
\end{proof}

\begin{remark}
Unlike in the scalar valued situation, an operator might be diagonalisable with different set of eigenvalues, and thus different eigenbases. This will lead to different matrix valued orthogonal polynomials that might be equivalent. As an example let us consider diagonal change of eigenbasis. Consider $f_i=P_ie_i$ and $\Psi_i=Q_i\phi_i$ with $P_i,\ Q_i\in GL_n(\dC)$. Then one has
\begin{equation}
H\cdot f_i =(P_i\Lambda P_i^{-1}) f_i=\Lambda'_i f_i, \quad P\cdot \Psi_j = (Q_j\theta_jQ_j^{-1})\Psi_j=\theta'_j\Psi_j. 
\end{equation}
We then introduce $\Pi'_n$ by 
\begin{equation}
( f_n, \Psi_j) = \Pi'_n(\theta'_j) ( f_0,\Psi_j) . 
\end{equation}
The matrix $( f_0,\Psi_j) $ is invertible if $( e_0,\phi_j)$ is invertible. This leads to the relation 
\begin{equation}
\Pi'_n(\theta'_j)=P_n\Pi_n(\theta_j)P_0^{-1}. 
\end{equation}
Thus the polynomials $\Pi_n$ and $\Pi'_n$ gives two equivalent families of matrix valued orthogonal polynomials in the sense that their associated weight are conjugated by an invertible matrix. More precisely one has: 
\begin{equation}
P_nP_m^* \delta_{nm} =\sum_j \Pi'_n(\theta'_j)W'(j)\Pi'_m(\theta'_j)^*,
\end{equation}
with $W'(j)=P_0W(j)P_0^{-1}$.  
\end{remark}

Back to generalities, we describe a setting that we are going to use to provide several examples in the last three sections. More precisely, using the Morita equivalence  we create operators acting tridiagonally on a $M_n(\dC)$-basis starting from operators acting tridiagonally on a $\dC$-basis. From now on, to avoid lengthy notations, we will denote by the same letter elements in $\End_\dC(V)$ and their image by the functor $G$.  

Let us consider a complex vector space $V$ of dimension $\dim V=n(N+1)$ together with a self-adjoint operator $H$ with eigenbasis $\ket{k}$ for the eigenvalues $\lambda_k$. Consider $P$ a self-adjoint operator acting tridiagonally on $\ket{k}$
\begin{equation}
P\ket{k}= a_{k+1}^* \ket{k+1}+b_k\ket{k}+a_k \ket{k-1},\quad k=0,\ldots,n(N+1)-1.
\end{equation}
with non zero $a_{k+1}$. We write $\ket{\phi_j}$ for the eigenvectors of $P$ with eigenvalues $\mu_j$.

Now let us consider the $M_n(\dC)$-module $V^n$, and the two $M_n(\dC)$-bases 
\begin{equation}
e_k=(\ket{nk},\ldots,\ket{nk+n-1}),\ \Phi_j=(\ket{\phi_{nj}},\ldots,\ket{\phi_{nj+n-1}}), \quad k,j=0,\ldots,N.
\end{equation}

\begin{proposition}
\label{prop:action-Pn-basis}
The operators $P^n$ act on the $M_n(\dC)$-bases $(e_k)_k,$ and $(\Phi_j)_j$ by 
\begin{align}
P^n\cdot \Phi_j &= \alpha_j \Phi_j, \qquad \alpha_j=\diag(\mu_{nj}^n, \ldots, \mu_{nj+n-1}^n ), \label{eq:alpha-P}\\
 P^n\cdot e_k&=A_{k+1}^\ast e_{k+1} + B_k e_k +A_k e_{k-1}. 
\end{align}
Moreover, $B_k=B_k^\ast$ for $k=0,\ldots, N,$ and $A_k$ is invertible and upper triangular for $k=1,\ldots,N$.  
\end{proposition}
\begin{proof}
    Since $P$ is a self-adjoint operator, so is $P^n$. Then Proposition \ref{pr:3TR-selfadjoint} tells us that the action of $P^n$ on $\Phi_j$ has the expected form if it is tridiagonal. The tridiagonality together with the fact that the $A_k$ are upper triangular is a consequence of the following formula which is obtained by induction
    \begin{equation}
    P^{n}\ket{k}=\sum_{\ell=-n}^{n}\alpha_{\ell,n}\ket{k+\ell},
    \end{equation}
    with  $\alpha_{n,n}=\prod_{i=1}^n a_{k+i}^*$ and $\alpha_{-n,n}=\prod_{i+0}^{n-1}a_{k-i}$. 
    The $j$th ($1\leq j\leq n$) coefficient on the diagonal of $A_k$ is then 
    \begin{equation}
    \prod_{i=0}^{n-1} a_{nk+j-1-i}\neq 0.
    \end{equation}
    This proves that $A_k$ is invertible. 
\end{proof}

\begin{remark}
    As seen in the proof, the explicit coefficients in $A_k$ and $B_k$ can be computed by induction from the action of $P$ on $\ket{k}$. However, the explicit expression for these coefficients quickly become difficult to handle when $n$ increase. 
\end{remark}
\begin{remark}
For the sake of simplicity we have used the operator $P^n$, but similar constructions can also be carried out by considering $q(P)$, where $q$ is a polynomial of degree $n$ with scalar coefficients. This provides a method for generating additional examples. An even more general situation arises by considering an operator $P$ acting $(2n+1)$-diagonally on $\ket{k}$, which is closely related to the ideas developed in \cite{DuranVanAssche95}. 
\end{remark}
\begin{remark}
In a probabilistic context, building on the foundational work of Karlin and McGregor, which established a deep connection between random walks with tridiagonal transition matrices and the theory of orthogonal polynomials, the authors in \cite{RomanMenchon} study an extension of this framework in which the transition matrix is given by a polynomial in a tridiagonal matrix, allowing transitions beyond nearest neighbors. The example provided in \cite{RomanMenchon} can be understood in our setting by considering an specific second degree polynomial $q(P)$, c.f. \cite[Sec 4.1]{RomanMenchon}. This example illustrates the significant potential of the algebraic framework developed in the present article.
\end{remark}

\begin{remark}
The case $n=1$ of the method simplifies to the scalar case \cite{GranovskiiZhedanov86,Zhedanov96}, where scalar orthonormal polynomials are recovered as transition coefficients: 
\begin{equation}
\label{eq:scalar-transition-coeff}
p_{k}(\mu_j)=\bra{k}\ket{\phi_j} w_j^{-\frac{1}{2}},\qquad \text{and}\quad  w_j=|\bra{0}\ket{\phi_j}|^2.
\end{equation}
\end{remark} 
\begin{proposition}\label{prop:Vandermonde}
\label{prop:L-R}Let $L(j)$ be the alternant matrix defined by 
\begin{equation}
\label{eq:alternant-general}
L(j)_{i,\ell}=p_{i-1}(\mu_{nj+\ell-1}), \qquad i,\ell=1,\ldots, n. 
\end{equation}
If $\mu_i\neq \mu_j$ for all $i\neq j$ then the following statements hold true:
\begin{enumerate}
    \item The matrix $L(j)$ is invertible.
    \item The family $(\theta_j)_j$ is given by $\theta_j=L(j)\alpha_j L(j)^{-1}$, where $\alpha_j$ is defined in \eqref{eq:alpha-P}, and satisfies the Vandermonde condition.
    \item The weight decomposes as $W(j)=L(j)D(j)L(j)^\ast,$ where $D(j)=\diag(w_{nj},\ldots,w_{nj+n-1}).$ 
    \item The matrix valued polynomials $R_k$ defined by
    $R_k(j)_{i,\ell}=p_{nk+i-1}(\mu_{nj+\ell-1}),$ $i,\ell=1,\ldots, n, $ 
    satisfy the orthogonality relation $\sum_j R_k(j)D(j)R_m(j)^\ast=\delta_{k,m}$.
    \item The polynomials $\Pi_k(x)$ satisfy $\Pi_k(\theta_j)=R_k(j)L(j)^{-1}.$
\end{enumerate}
\end{proposition}
\begin{proof}
\begin{enumerate}
\item The matrix $L(j)$ is an alternant matrix in which the $i$th row is a polynomial of degree  $i-1$ in $\mu_j$. It is known see \cite{Muir} that
$$\det L(j)=\beta_n \det V(\mu_{nj},\ldots, \mu_{nj+n-1}),$$
where $\beta_n$ is the product of the leading coefficients of the  polynomials $p_i$, and $V(\mu_{nj},\ldots, \mu_{nj+n-1})$ represents the Vandermonde matrix, which is invertible since $\mu_j\neq\mu_k$ for $j\neq k$.
\item The decomposition of the weight and the family $\theta_j$ follows directly from  the fact that the entries of the matrix $(e_0,\Phi_j)$ are given by
\begin{equation}
(e_0,\Phi_j)_{i,\ell}=\bra{i-1}\ket{\phi_{nj+\ell-1}}=p_{i-1}(\mu_{nj+\ell-1})h_{i-1}^{-\frac12}w^\frac12_{nj+\ell-1},    
\end{equation}
i.e., $(e_0,\Phi_j)=L(j)\diag(w_{nj}^\frac12,\ldots,w_{nj+n-1}^\frac12).$

The Vandermonde matrix $V(\theta_0,\ldots,\theta_j)$ can be rewritten as
\begin{equation}
\begin{pmatrix}I & I & \cdots & I \\
\theta_0 & \theta_1& \cdots & \theta_j \\
\vdots & \vdots & \ddots & \vdots \\
\theta_0^j &\theta_1^j & \cdots & \theta_j^j
\end{pmatrix}=\begin{pmatrix}L(0) & L(1)& \cdots & L(j) \\
L(0)\alpha_0 & L(1)\alpha_1& \cdots & L(j)\alpha_j \\
\vdots & \vdots & \ddots & \vdots \\
L(0)\alpha_0^j &L(1)\alpha_1^j & \cdots & L(j)\alpha_j^j
\end{pmatrix}\begin{pmatrix}
L(0)^{-1}&0&\cdots&0\\
0& L(1)^{-1}&\cdots&0\\
\vdots&\vdots&\ddots&\vdots\\
0&0&\cdots&L(j)^{-1}
\end{pmatrix}\,,
\end{equation}
thus $V(\theta_0,\ldots,\theta_j)$ being invertible is equivalent to 
\begin{equation}
\begin{pmatrix}L(0) & L(1)& \cdots & L(j) \\
L(0)\alpha_0 & L(1)\alpha_1& \cdots & L(j)\alpha_j \\
\vdots & \vdots & \ddots & \vdots \\
L(0)\alpha_0^j &L(1)\alpha_1^j & \cdots & L(j)\alpha_j^j
\end{pmatrix}
\end{equation}
being invertible. Let us show that the columns of the above $(nj+n)\times (nj+n)$ matrix are linearly independent. Thus we choose some scalars $a_\ell$ and the linear combination of the columns writes
\begin{align}
\sum_{\ell=0}^{nj+n-1}a_\ell p_s(\mu_\ell)\mu_\ell^{nk}&=0, \qquad s=0,\ldots,n-1,  \quad k=0,\ldots,j.
\end{align}
The above set of equations is equivalent to 
\begin{align}
\sum_{\ell=0}^{nj+n-1}a_\ell\mu_{\ell}^{nk+s}=0, \qquad s=0,\ldots,n-1,\quad k=0,\ldots,j.
\end{align}
Writing these equations in matrix form lead to 
\begin{equation}
\begin{pmatrix}
1&\ldots&1\\
\mu_0&\ldots&\mu_{nj+n-1}\\
\vdots&\ddots&\vdots\\
\mu^{nj+n-1}_0&\ldots&\mu^{nj+n-1}_{nj+n-1}\end{pmatrix}\begin{pmatrix}
a_0\\
a_1\\
\vdots\\
a_{nj+n-1}
\end{pmatrix}=\begin{pmatrix}
0\\
0\\
\vdots\\
0
\end{pmatrix}\,,
\end{equation}
and we conclude that $a_i=0$ for all $i=0,\ldots, nj+n-1$ since $\mu_j\neq \mu_k$ for $j\neq k$. Thus $V(\theta_0,\ldots,\theta_j)$ is invertible and the Vandermonde condition is satisfied. 
\item The orthogonality of the polynomials $(R_k)_k$ follows by noticing 
\begin{align}
(\sum_jR_k(j)D(j)R_m(j)^\ast)_{i,\ell}&=\sum_j\sum_{s=1}^np_{nk+i-1}(\mu_{nj+s-1})p_{nm+\ell-1}(\mu_{nj+s-1})w_{nj+s-1}\nonumber\\
    &=\sum_rp_{nk+i-1}(\mu_r)p_{nm+\ell-1}(\mu_r)w_r\nonumber\\
    &=\delta_{nk+i,nm+\ell}.
\end{align}
Since $nk+1\leq nk+i\leq n(k+1),$ $nm+1\leq nm+\ell\leq n(m+1)$, if $k<m$ we have $n(k+1)<nm+1$ and then $\delta_{nk+i,nm+\ell}=\delta_{k,m}\delta_{i,\ell}.$ 
\item Finally, since
\begin{equation}
    (e_k,\phi_j)_{i,\ell}=\bra{nk+i-1}\ket{\phi_{nj+\ell-1}}=p_{nk+i-1}(\mu_{nj+\ell-1})w_{nj+\ell-1}^\frac12,
\end{equation}
we get
\begin{equation}
    \Pi_k(\theta_j)=(e_k,\phi_j)(e_0,\phi_j)^{-1}=R_k(j)L(j)^{-1}.
\end{equation}
This concludes the proof of the proposition. 
\end{enumerate}
\end{proof}
Despite having a simpler orthogonality relation we remark that the polynomial $R_0$ is not constant, and thus the family $R_k$ does not define a $M_n(\dC)$-basis of the module of matrix valued orthogonal polynomials.

%%%%%%%%%%%%%%%%%%%%%%%%%%%%%%%%%
%%%%%%%%%%%%%%%%%%%%%%%%%%%%%%%%%
%%%%%%%%%%%%%%%%%%%%%%%%%%%%%%%%%
%%%%%%%%%%%%%%%%%%%%%%%%%%%%%%%%%

\section{Krawtchouk type polynomials}
In \cite{GranovskiiZhedanov86}, Krawtchouk polynomials are expressed as transition coefficients between eigenbases related to the Lie algebra $\mathfrak{su}(2)$. Starting from this fact we construct matrix analogues of Krawtchouk polynomials. Generators of $\mathfrak{su}(2)$ satisfy the commutation relations
\begin{equation}
    [H,L^+]=2L^+,\qquad [H,L^-]=-2L^-,\qquad [L^+,L^-]=H.
\end{equation}
Its $m+1$ dimensional representation can be described as follows on the vector space $V$ with basis $\ket{k}$
\begin{align}
H \ket{k}=\lambda_k\ket{k},&\qquad \lambda_k=2k-m,\qquad k=0,\ldots,m, \\
L^{+} \ket{k}=\rho^+_k\ket{k+1},&\qquad \rho^+_k=\sqrt{(k+1)(m-k)},\\
L^{-} \ket{k}=\rho^-_k\ket{k-1},&\qquad \rho^-_k=\sqrt{k(m+1-k)},
\end{align}
observe that $\rho_k^+=\rho_{k+1}^-.$ Clearly, the operator $P=\cos(a)H+\sin(a)(L^++L^-)$ acts tridiagonally on $\ket{k},$ 
\begin{equation}
    P\ket{k}=\sin(a) \rho_k^+ \ket{k+1}+ \cos(a)\lambda_k\ket{k}+\sin(a)\rho_k^-\ket{k-1}.
\end{equation}
It is known that its spectrum is discrete 
\begin{equation}
\label{eq:spectrum-P}
P\ket{\phi_j}=\mu_j\ket{\phi_j},\qquad \mu_j=2j-m,\qquad j=0,\ldots,m, 
\end{equation}
and that the operator $H$ acts tridiagonally on $\ket{\phi_j}$
\begin{equation}
    H \ket{\phi_j}=\sin(a)\rho^+_j\ket{\phi_{j+1}}+\cos(a)\lambda_j\ket{\phi_j}+\sin(a)\rho^-_{j}\ket{\phi_{j-1}}.
\end{equation}
Krawtchouk polynomials are defined by
\begin{equation}
K_{k}(j;p,m)=\rFs{2}{1}{-j,-k}{-m}{\frac{1}{p}},
\end{equation}
and can be recovered as transition coefficients between the two bases $\ket{k}$ and $\ket{\phi_j}$
\begin{equation}
\label{eq:scalar-krawtch-transition-coeff}
K_{k}(j;p,m)=\bra{k}\ket{\phi_j}h_k^\frac12 w_j^{-\frac{1}{2}},\qquad j,k=0,\ldots,m,
\end{equation}
where
\begin{equation}
\label{eq:weight-norm-scalar-krawt}
w_j=\binom{m}{j}p^j(1-p)^{m-j}, \qquad h_k=\frac{(-1)^kk!}{(-m)_k}\left(\frac{1-p}{p}\right)^k, \qquad \text{and}\qquad p=\frac{1}{2}(1+\cos(a)).
\end{equation}

%%%%%%%%%%%%%%%%%%%%%%%%%%%
%%%%%%%%%%%%%%%%%%%%%%%%%%%
%%%%%%%%%%%%%%%%%%%%%%%%%%%
%%%%%%%%%%%%%%%%%%%%%%%%%%%
\subsection{A $n\times n$ Krawtchouk type}
\label{sec:nxn-example-krawtchouk}
We consider a representation $V$ of dimension $n(N+1)$, i.e. we take $m=nN+n-1$. By Corollary \ref{cor:free-ndivdimV}, the $M_n(\dC)$-module $V^n$ is free and we have the following $M_n(\mathbb{C})$-bases for $V^n$
\begin{equation}
    e_k=(\ket{nk},\ldots,\ket{nk+n-1}), \qquad \Phi_j=(\ket{\phi_{nj}}, \ldots, \ket{\phi_{nj+n-1}}), \qquad k,j=0,\ldots,N.
\end{equation}
The actions of the operators $P^n$ and $H^n$ on the $M_n(\dC)$-bases $(e_k),$ and $(\Phi_j)$ are described in Proposition \ref{prop:action-Pn-basis}. As discussed in Section \ref{sec:algebraic-interpretation-mvop}, the construction of MVOPs can proceed provided the matrices $(e_0,\Phi_j)$ are invertible. We can verify this by examining the entries of the matrix, which are given by
\begin{equation}
\label{eq:e0-phij-alternant}
(e_0,\Phi_j)_{i,\ell}=\bra{i-1}\ket{\phi_{nj+\ell-1}} =K_{i-1}(nj+\ell-1;p,nN+n-1)h_{i-1}^{-\frac12}w^\frac12_{nj+\ell-1} \qquad i,\ell=1,\ldots, n.
\end{equation}
As shown in Proposition \ref{prop:L-R}, the $n\times n$ matrix $L(j)$ is defined as
\begin{equation}
\label{eq:alternant-krawtchouk}
L(j)_{i,\ell}=h_{i-1}^{-\frac12}K_{i-1}(nj+\ell-1;p,nN+n-1), \qquad i,\ell=1,\ldots, n, \quad j=0,\ldots,N,
\end{equation}
is an invertible alternant matrix. 
%Each row of $L(j)$ is a polynomial of degree  $i-1$ in $j$. It is known see \cite{Muir} that
%\begin{equation}
%    \det L(j)=p_n \det V(nj,\ldots, nj+n-1),
%\end{equation}
%where $p_n$ is the product of the leading coefficients and the inverse of the norms of the Krawtchouk polynomials, and $V(nj,\ldots, nj+n-1)$ represents the Vandermonde matrix. The determinant of $V$ is given by
%\begin{equation}
%    \det V(nj,\ldots, nj+n-1)=(n-1)!(n-2)!\ldots 2! 1!.
%\end{equation}    
Moreover, from equation \eqref{eq:e0-phij-alternant} we can express
\begin{equation}
    (e_0,\Phi_j)=L(j)\diag(w^\frac12_{nj} ,\ldots,w^\frac12_{nj+n-1}),
\end{equation}
and since the matrix $L(j)$ is invertible, we conclude that  $( e_0,\Phi_j)$ is also invertible.
Since the matrices $A_k$ constructed from Proposition \ref{prop:action-Pn-basis} are also invertible, we can proceed as outlined in Section \ref{sec:algebraic-interpretation-mvop} and obtain the family of MVOPs $\Pi_k(x)$ satisfying the recurrence relation \begin{equation}
\Pi_k(\theta_j)\theta_j=A_{k+1}^*\Pi_{k+1}(\theta_j)+B_k\Pi_k(\theta_j)+A_k\Pi_{k-1}(\theta_j), \qquad 
\Pi_{-1}(\theta_j)=0,\qquad \Pi_0(\theta_j)=I_n,
\end{equation}
where $\theta_j=( e_0, \Phi_j) \alpha_j^*( e_0, \Phi_j)^{-1}.$
Additionally, the following orthogonality relation holds
\begin{equation}
\label{eq:orth-krawtchouk-nxn}
\sum_{j=0}^N \Pi_{k}(\theta_j)W(j)\Pi_{\ell}(\theta_j)^\ast=\delta_{k,\ell},\qquad W(j)=L(j)D(j)L(j)^\ast,
\end{equation}
where the matrix $L(j)$ is defined in \eqref{eq:alternant-krawtchouk}, $D(j)=\diag(w_{nj}, \ldots, w_{nj+n-1}),$ $w_j$ as in \eqref{eq:weight-norm-scalar-krawt} for $m=nN+n-1$. Moreover, the polynomials satisfy the difference equation 
\begin{multline}
\Lambda_k\Pi_k(\theta_j)=\Pi_k(\theta_{j+1})(e_0,\Phi_{j+1})A_{j+1}(e_0,\Phi_j)+\Pi_k(\theta_{j})(e_0,\Phi_{j})B_{j}(e_0,\Phi_j)\\+\Pi_k(\theta_{j-1})(e_0,\Phi_{j-1})A_{j}^\ast(e_0,\Phi_j). 
\end{multline}

We conclude this section by presenting explicit expressions in the $2\times2$ case. The alternant matrix $L(j)$ from equation \eqref{eq:alternant-krawtchouk} is
\begin{equation}
    L(j)=\begin{pmatrix}
1&1\\
h_{1}^{-\frac{1}{2}}K_1(2j;p,2N+1)&h_{1}^{-\frac{1}{2}}K_1(2j+1;p,2N+1)
\end{pmatrix},
\end{equation}
and its inverse is 
\begin{equation}
    L(j)^{-1}=-p(2N+1)\begin{pmatrix}
K_1(2j+1;p,2N+1)&-h_{1}^{\frac{1}{2}}\\
-K_1(2j;p,2N+1)&h_{1}^{\frac{1}{2}}
\end{pmatrix}.
\end{equation}
The weight function is $W(j)=L(j)\diag(w_{2j},w_{2j+1})L(j)^\ast.$ The matrix $A_k$ on the recurrence relation and difference equation is given by
\begin{equation}
    A_{k}= \begin{pmatrix}
\sin(a)^2\rho^-_{2 k} \rho^-_{2 k-1} & \sin(a)\cos(a)\rho^-_{2 k}(\lambda_{2 k}+\lambda_{2 k -1})  
\\
0 & \sin(a)^2\rho^-_{2 k}\rho^-_{2 k +1}
\end{pmatrix},
\end{equation}
and the matrix $B_k$ is given by
\begin{equation}
    B_k=\begin{pmatrix}
\sin(a)^2 \rho_{2 k}^{-2} +\sin(a)^2 \rho _{2 k+1}^{-2} +\cos(a)^{2} \lambda_{2 k}^{2}& \sin(a)\cos(a)(\lambda_{2 k} \rho_{2 k+1}^- + \rho_{2 k+1}^-\lambda_{2 k +1})
\\
 \sin(a)\cos(a)(\lambda_{2 k} \rho_{2 k+1}^- + \rho_{2 k+1}^-\lambda_{2 k +1}) & \sin(a)^2 \rho_{2 k +1}^{-2}+\sin(a)^2\rho_{2 k+2}^{-2}+\cos(a)^2 \lambda_{2 k +1}^{2}
\end{pmatrix},
\end{equation}
where $
\lambda_k=2k-2N-1$ and $ \rho^-_k=\sqrt{k(2N+2-k)}$.
Finally, the MVOPs are given by $\Pi_k(\theta_j)=R_k(j)L(j)^{-1},$ where 
\begin{equation}
    R_k(j)=\begin{pmatrix}
K_{2k}(2j)h_{2k}^{-\frac{1}{2}}&K_{2k}(2j+1)h_{2k}^{-\frac{1}{2}}\\
K_{2k+1}(2j)h_{2k+1}^{-\frac{1}{2}}&K_{2k+1}(2j+1)h_{2k+1}^{-\frac{1}{2}}
\end{pmatrix},
\end{equation}
where $h_k$ is defined in \eqref{eq:weight-norm-scalar-krawt}, $m=2N+1$ and $\theta_j=L(j)\diag(\mu_{2j}^2,\mu_{2j+1}^2)L(j)^{-1}$.

%%%%%%%%%%%%%%%%%%%%%%%%%%%
%%%%%%%%%%%%%%%%%%%%%%%%%%%
%%%%%%%%%%%%%%%%%%%%%%%%%%%
%%%%%%%%%%%%%%%%%%%%%%%%%%%

\subsection{Another $2\times 2$ Krawtchouk type }

In this subsection we exhibit another set of generators which leads to another Krawtchouk $2\times 2$ matrix analog. We consider $V$ a $2(N+1)$ dimensional vector space so that the $M_2(\dC)$-module $V^2$ is free. Let us consider the $M_2(\dC)$-bases of $V^2$ given by
\begin{equation}
    f_k=(\ket{k},\ket{k+N+1}), \qquad \Phi_j=(\ket{\phi_{j}}, \ket{\phi_{j+N+1}}),\quad j,k=0,\ldots,N.
\end{equation}
Notice that considering vectors of the form
\begin{equation}
        f_k=(\ket{k},\ket{k+N+1},\cdots, \ket{k+(n-1)(N+1)}),\qquad j,k=0,\ldots,N,
\end{equation}
and similarly for $\Phi_j$, will lead to a similar $n\times n $ example. However, for sake of simplicity we stick to the $n=2$ case.

The operators $P^2$ and $H^2$ act on $\Phi_j$ and $f_k$ as described in Proposition \ref{prop:action-Pn-basis}
where the matrices $A_k$ and $B_k$ are explicitly given by
\begin{equation}
A_k=\begin{pmatrix}
\rho^-_{k}&0\\0&\rho^-_{k+N+1}\end{pmatrix}, \qquad B_k=\begin{pmatrix}
\lambda_{k}&0\\0&\lambda_{k+N+1} \end{pmatrix}.    
\end{equation}

The matrix $(f_0,\Phi_j)$ is given by 
\begin{equation}
(f_0,\Phi_j)=\begin{pmatrix}
1&1\\
K_{N+1}(j)h_{N+1}^{-\frac{1}{2}}&K_{N+1}(j+N+1)h_{N+1}^{-\frac{1}{2}}
\end{pmatrix} \begin{pmatrix}
w_{j}^{\frac{1}{2}}&0\\
0&w_{j+N+1}^{\frac{1}{2}}
\end{pmatrix},
\end{equation}
where $K_{N+1}(j)=K_{N+1}(j;p;2N+1)$. Using the hypergeometric expansion of Krawtchouk polynomials, one checks that $K_{N+1}(j+N+1)-K_{N+1}(j)$ is a non zero polynomial in $1/p$. Thus, except for possibly a finite number of $p$, the alternant matrix 
\begin{equation}
\label{eq:L-2nd-Krawt}
L(j)=\begin{pmatrix}
1&1\\
K_{N+1}(j)h_{N+1}^{-\frac{1}{2}}&K_{N+1}(j+N+1)h_{N+1}^{-\frac{1}{2}}
\end{pmatrix},
\end{equation}
 is invertible, and so is $(f_0,\Phi_j)$ for all $j=0,\ldots,N$. Notice that numerical computation showed that there exists $0<p<1$ such that $K_{N+1}(j+N+1)-K_{N+1}(j)=0$ thus we have to assume it is not the case. The inverse of $L(j)$ is then given by 
\begin{equation}
L(j)^{-1}=\frac{1}{K_{N+1}(j+N+1)-K_{N+1}(j)}\begin{pmatrix}
K_{N+1}(j+N+1)&-h_{N+1}^{\frac{1}{2}}\\
-K_{N+1}(j)&h_{N+1}^{\frac{1}{2}}
\end{pmatrix}.
\end{equation}
Using similar arguments than for Proposition \ref{prop:Vandermonde}, up to a permutation of the lines, one can check that the family $\theta_j=L(j)\alpha_jL(j)^-1$ satisfy the Vandermonde condition, where $\alpha_j=\begin{pmatrix}
    \mu_j&0\\0&\mu_{j+N+1}
\end{pmatrix}$. We can thus proceed as in Section \ref{sec:algebraic-interpretation-mvop} and obtain the family $\Pi_k(x)$ of MVOPs, satisfying the recurrence relation 
 \begin{equation}
\Pi_k(\theta_j)\theta_j=A_{k+1}^*\Pi_{k+1}(\theta_j)+B_k\Pi_k(\theta_j)+A_k\Pi_{k-1}(\theta_j), \qquad 
\Pi_{-1}(\theta_j)=0,\qquad \Pi_0(\theta_j)=I_n,
\end{equation}
where $\theta_j=( f_0, \Phi_j) \theta_j'( f_0, \Phi_j)^{-1}.$ Additionally the polynomials satisfy  the difference equation 
\begin{multline}
\Lambda_k\Pi_k(\theta_j)=\Pi_k(\theta_{j+1})(f_0,\Phi_{j+1})A_{j+1}(f_0,\Phi_j)+\Pi_k(\theta_{j})(f_0,\Phi_{j})B_{j}(f_0,\Phi_j)\\+\Pi_k(\theta_{j-1})(f_0,\Phi_{j-1})A_{j}^\ast(f_0,\Phi_j). 
\end{multline}
Moreover, the following orthogonality relation holds
\begin{equation}
    \sum_{j=0}^N \Pi_{k}(\theta_j)W(j)\Pi_{\ell}(\theta_j)^\ast=\delta_{k,\ell}, \qquad W(j)=L(j)D(j)L(j)^\ast,
\end{equation} 
where $L(j)$ is defined in \eqref{eq:L-2nd-Krawt}, $D(j)=\diag(w_j,w_{j+N+1})$
and $\theta_j=L(j)\diag(\mu_j,\mu_{j+N+1})L(j)^{-1}.$

\begin{remark}
The Krawtchouk type example given in this section is not equivalent to the one given in Section \ref{sec:nxn-example-krawtchouk}. A proof of this fact can be carried out by showing that if $P$ is a constant matrix such that $W_2(x)=PW_1(x)P^\ast$ then $P=0.$
\end{remark}

%%%%%%%%%%%%%%%%%%%%%%%%%%%%%%%%%
%%%%%%%%%%%%%%%%%%%%%%%%%%%%%%%%%
%%%%%%%%%%%%%%%%%%%%%%%%%%%%%%%%%
%%%%%%%%%%%%%%%%%%%%%%%%%%%%%%%%%

\section{Discrete Chebyshev type polynomials}
\label{discrete Chebyshev}
We construct a last example to show that the procedure developed in this article is also suited to create $q$ deformed MVOPs. This will be based on the $q$-deformed Lie algebra $\mathfrak{so}_q(3)$ and the work done in \cite{Zhedanov96}.  We consider $N$ a positive integer, $\omega=\frac{\pi}{N}$, and  $\mathfrak{so}_q(3)$ at $q=e^{2i\omega}$  a root of unity. The generators of this $q$-deformed algebra satisfy the commutation relations
\begin{equation}
    [K_0,K_1]_\omega=K_2, \quad [K_1,K_2]_\omega=-K_0,\quad [K_0,K_2]_\omega=K_1,
\end{equation}
where $[A,B]_\omega = e^{\frac{i\omega}{2}}AB-e^{-\frac{i\omega}{2}}BA,$ denotes the symmetric $q$-commutator. For $d = 1, 2,\ldots,N-1$, there is a unitary $(d+1)$-dimensional  irreducible representation $(V,\rho)$ of $\mathfrak{so}_q(3)$  in which the operator $K_0$ is diagonal and $K_1$ two diagonal 
\begin{align}
&K_0 \ket{k}=\lambda_k\ket{k}, \quad k=0,1,\ldots,d, \\
&K_1 \ket{k}= \rho_{k+1}\ket{k+1}+\rho_k\ket{k-1},
\end{align}
where
\begin{equation}
\lambda_k=\frac{\cos\omega(k+\beta)}{\sin \omega},\qquad \text{and}\qquad \rho_k=\sqrt{\frac{\sin \omega k \sin\omega(k+2 \beta-1)}{4\sin^2\omega \sin\omega(k+\beta)\sin \omega (k+\beta-1)}},
\end{equation}
with $2\beta=N-d$. 
Given $b\in\mathbb{R}$, we consider the operator $P=K_1+b$ which acts tridiagonally on $\ket{k}$
\begin{equation}
    P\ket{k}=\rho_{k+1}\ket{k+1}+b\ket{k}+\rho_k\ket{k-1}.
\end{equation}
The operator $K_1$ has the same spectrum as $K_0$ \cite{Zhedanov96,Crampe2021}, and therefore $P$ is diagonalized as follows
\begin{equation}
    P\ket{\phi_j}=\mu_j\ket{\phi_j}, \qquad k=0,1,\ldots,d, \quad \mu_j=\frac{\cos \omega(j+\beta)}{\sin\omega}+b,
\end{equation}
The operator $H=K_0$ acts tridiagonally on $\ket{\phi_j}$ 
\begin{equation}
H\ket{\phi_j}=\rho_{j+1}\ket{\phi_{j+1}}+\rho_j\ket{\phi_{j-1}}.
\end{equation}
The monic $q$-ultraspherical polynomials at $q$ a root of unity can be recovered as the overlaps
\begin{equation}
    P_k(x_j)=\bra{k}\ket{\phi_j}h_k^{\frac{1}{2}}w_j^{-\frac{1}{2}},\qquad x_j=2\cos\omega(j+\beta)
\end{equation}
where
\begin{equation}
\label{eq:weight-norm-scalar-q-cheb}
w_j=\sin\omega(j+\beta)\prod_{\ell=1}^{2\beta-1}\sin \omega (j+\ell),\qquad  s=0,\ldots ,d .
\end{equation}
The formulas for $h_k$ are more involved and can be found in \cite{Zhedanov96} (formulas (30), (31)). 

It is interesting to realize that the finite Chebyshev are a special case of these
$q$-polynomials, taking $\beta=1$ and with $\rho_{k+1}^+$ becoming $\frac{1}{2}$. Moreover, the finite Chebyshev polynomials are given by
\begin{equation}
    P_n(x_j)=\frac{\sin \omega (n+1)(j+1)}{\sin \omega (j+1)}, \qquad j=0,1,\ldots, d,
\end{equation}
and the weight function is 
\begin{equation}
    w_j=\sin^2 \omega(j + 1).
\end{equation}

\subsection{A $2\times 2$ discrete Chebyshev type}
We now construct $2\times 2$ discrete Chebyshev type polynomials, which can be generalized to $n\times n$ as in Section \ref{sec:nxn-example-krawtchouk} and Section \ref{sec:nxn-example-meixner}. We consider the $(d+1)$-dimensional irreducible representation $(V,\rho)$ of $\mathfrak{so}_q(3)$, and we assume $2|d+1,$ i.e., $d+1=2(m+1).$ 
We consider the $M_2(\mathbb{C})$ basis of $V^2$ given by 
\begin{equation}
    e_k=(\ket{2k},\ket{2k+1}), \qquad \Phi_j=(\ket{\phi_{2j}}, \ket{\phi_{2j+1}}), \qquad k,j=0,\ldots,m.
\end{equation}
The action of the operators $P^2$ and $H^2$ on the $M_n(\dC)$-bases $(e_k),$ and $(\Phi_j)$ is outlined in Proposition \ref{prop:action-Pn-basis}. Specifically, we have the following:
\begin{equation}
    P^2\cdot \Phi_j = \alpha_j \Phi_j,\qquad H^2\cdot e_k=\Lambda_k e_k,
\end{equation}
where $\theta'_j=\diag(\mu_{2j}^2,\mu_{2j+1}^2 ),$ $\Lambda_k=\diag(\lambda_{2j}^2, \lambda_{2j+1}^2 )$ and
\begin{align}
 P^2\cdot e_k&=A_{k+1}^\ast e_{k+1} + B_k e_k +A_k e_{k-1},\\
 H^2\cdot \Phi_j &= A_{k+1}^\ast\Phi_j+B_k\Phi_j+A_k\Phi_{j-1}.
\end{align}
Here, the matrix $A_k,$ is invertible for $k=1,\ldots,N,$ $B_k=B_k^\ast.$ The matrices $A_k$ and $B_k$ are given by
\begin{align}
A_{k}&= \begin{pmatrix}
\rho_{2 k} \rho_{2 k-1} & 2b\rho_{2 k}
\\
0 & \rho_{2 k}\rho_{2 k +1}
\end{pmatrix}, \qquad k=1,\ldots,N,
\\
B_k&=\begin{pmatrix}
 \rho_{2 k}^2 + \rho _{2 k+1}^2 +b^{2}& 2b\rho_{2 k+1} 
\\
 2b\rho_{2 k+1} &  \rho_{2 k +1}^2+ \rho_{2 k+2}^2+b^{2}
\end{pmatrix}\qquad k=0,\ldots,N,
\end{align}
Moreover, in the discrete Chebyshev case $A_k$ and $B_k$ are independent of $k$ and are given by
\begin{equation}
    A=\begin{pmatrix}
\frac{1}{4}&-b\\
0&\frac{1}{4}
\end{pmatrix},\qquad \text{and}\qquad B=\begin{pmatrix}
\frac{1}{2}+b^2&-b\\
-b&\frac{1}{2}+b^2
\end{pmatrix}.
\end{equation}
The construction of MVOPs can be carried out provided the matrices $(e_0,\Phi_j)$ are invertible. We can verify that this is so by noting that their entries are given by 
\begin{equation}
    (e_0,\Phi_j)_{i,\ell}=\bra{i-1}\ket{\phi_{nj+\ell-1}} = P_{i-1}(x_{nj+\ell-1})h_{i-1}^{-\frac{1}{2}}w_{nj+\ell-1}^\frac12 \qquad i,\ell=1,2 
\end{equation}
which implies that
\begin{equation}
    (e_0,\Phi_j)=L(j)\diag(w_{2j}^\frac12,w_{2j+1}^\frac12),
\end{equation}
where $L(j)$ is the alternant matrix 
\begin{equation}\label{eq:alternant-Chebyshev}
L(j)=\begin{pmatrix}h_{0}^{-\frac{1}{2}}&0\\
0&h_{1}^{-\frac{1}{2}}
\end{pmatrix}\begin{pmatrix}
1&1\\
P_1(x_{2j})&P_1(x_{2j+1})
\end{pmatrix}.
\end{equation}
Observe that since $\omega=\frac{\pi}{N},$ and $\beta \in \dN$ the matrix $L(j)$ is invertible. Thus $(e_0,\phi_j)$ is invertible.  The inverse of $L(j)$ is 
\begin{equation}
    L(j)^{-1}=\frac{1}{2(\cos\omega(2j+\beta)-\cos\omega(2j+1+\beta))}\begin{pmatrix}
1&1\\
P_1(x_{2j})&P_1(x_{2j+1})
\end{pmatrix}.\begin{pmatrix}h_{0}^{-\frac{1}{2}}&0\\
0&h_{1}^{-\frac{1}{2}}
\end{pmatrix}.
\end{equation}
We then proceed as in Section \ref{sec:algebraic-interpretation-mvop} and obtain the family of MVOPs $\Pi_k(x)$ that satisfy the recurrence relation 
 \begin{equation}
\Pi_k(\theta_j)\theta_j=A_{k+1}^*\Pi_{k+1}(\theta_j)+B_k\Pi_k(\theta_j)+A_k\Pi_{k-1}(\theta_j), \qquad 
\Pi_{-1}(\theta_j)=0,\qquad \Pi_0(\theta_j)=I_n,
\end{equation}
 where $\theta_j=( e_0, \Phi_j) \alpha_j^*( e_0, \Phi_j)^{-1}.$ Additionally, the polynomials satisfy the orthogonality relation
\begin{equation}
    \sum_{j=0}^N \Pi_{k}(\theta_j)W(j)\Pi_{\ell}(\theta_j)^\ast=\delta_{k,\ell},\qquad W(j)=L(j)D(j)L(j)^\ast,
\end{equation} 
where the matrix $L(j)$ is defined in \eqref{eq:alternant-Chebyshev} and $D(j)=\diag(w_{2j}, w_{2j+1}).$ These polynomials moreover satisfy the difference equation
\begin{multline}
   \Lambda_k\Pi_k(\theta_j)=\Pi_k(\theta_{j+1})(e_0,\Phi_{j+1})A_{j+1}(e_0,\Phi_j)+\Pi_k(\theta_{j})(e_0,\Phi_{j})B_{j}(e_0,\Phi_j)\\+\Pi_k(\theta_{j-1})(e_0,\Phi_{j-1})A_{j}^\ast(e_0,\Phi_j). 
\end{multline}
Finally, the MVOPs are given by $\Pi_k(\theta_j)=R_k(j)L(j)^{-1},$ where
\begin{equation}
    R_k(j)=\begin{pmatrix}
h_{2k}^{-\frac{1}{2}}P_{2k}(x_{2j})&h_{2k}^{-\frac{1}{2}}P_{2k}(x_{2j+1})\\
h_{2k+1}^{-\frac{1}{2}}P_{2k+1}(x_{2j})&h_{2k+1}^{-\frac{1}{2}}P_{2k+1}(x_{2j+1})
\end{pmatrix}.
\end{equation}
%%%%%%%%%%%%%%%%%%%%%%%%%%%%%%
%%%%%%%%%%%%%%%%%%%%%%%%%%%%%%
%%%%%%%%%%%%%%%%%%%%%%%%%%%%%%%

\section{Meixner type polynomials}
While the theory developed in this article focuses on finite-dimensional representations, we present an instructive infinite-dimensional example, noting that the theoretical framework necessary for its rigorous treatment lies beyond the scope of this paper and would require an adequate treatment of unbounded operators. We consider the Lie algebra $\mathfrak{su}(1,1)$ which will give matrix valued analogues of Meixner polynomials. Its generators satisfy the commutation relations
\begin{equation}
    [H,L^+]=L^+, \qquad [H,L^-]=-L^-, \qquad [L^-,L^+]=2H.
\end{equation}
We consider a discrete series representation $V$, which is described on the basis $\ket{k}$ by
\begin{align}
H \ket{k}=\lambda_k\ket{k},&\qquad \lambda_k=\frac{\beta}{2}+k, \qquad \beta>0,\quad k=0,1,\ldots \\
L^{+} \ket{k}=\rho^+_k\ket{k+1},&\qquad \rho^+_k=\sqrt{(k+1)(\beta+k)},\\
L^{-} \ket{k}=\rho^-_k\ket{k-1},&\qquad \rho^-_k=\sqrt{k(\beta+k-1)}.
\end{align}
As in the previous section we have $\rho^+_k=\rho_{k+1}^-.$ Given $a,b\in\mathbb{R}$, we consider the operator $P=-\frac{1}{2}\sinh(a)(L^++L^-)+\cosh(a)H$ which acts tridiagonally on $\ket{k}$
\begin{equation}
    P\ket{k}=-\frac{1}{2}\sinh(a) \rho_k^+ \ket{k+1}+\cosh(a) \lambda_k\ket{k}-\frac{1}{2}\sinh(a)\rho_k^-\ket{k-1}.
\end{equation}
The spectrum of the operator $P$ is 
\begin{equation}
\label{eq:spectrum-P-meixner}
P\ket{\phi_j}=\mu_j\ket{\phi_j},\qquad \mu_j=\frac{\beta}{2}+j,\qquad j=0,1\ldots
\end{equation}
and the operator $H$ acts tridiagonally on $\ket{\phi_j}$
\begin{equation}
    H\ket{\phi_j}=-\frac{1}{2}\sinh(a)\rho^+_j\ket{\phi_{j+1}}+\cosh(a)\lambda_j\ket{\phi_j}-\frac{1}{2}\sinh(a)\rho_j^-\ket{\phi_{j-1}}.
\end{equation}
Meixner polynomials are defined by
\begin{equation}
M_{k}(j;\beta,c)=\rFs{2}{1}{-j,-k}{\beta}{1-\frac{1}{c}},
\end{equation}
and can be recovered as transition coefficients between this basis
\begin{equation}
    M_{k}(j;\beta,c)=\bra{k}\ket{\phi_j}h_k^\frac{1}{2}w_j^{-\frac{1}{2}},\qquad  j,k=0,1,\ldots
\end{equation} 
where
\begin{equation}
\label{eq:weight-norm-scalar-meixner}
w_j=\frac{(\beta)_j}{j!}c^j, \qquad h_k=\frac{k!}{(\beta)_k(1-c)^{\beta} c^k},\qquad \text{and}\qquad c=\left(\frac{\cosh(a)-1}{\sinh(a)}\right)^2.
\end{equation}

%%%%%%%%%%%%%%%%%%%%%%%%%%%%%%%%%%%%
%%%%%%%%%%%%%%%%%%%%%%%%%%%%%%%%%%%%
%%%%%%%%%%%%%%%%%%%%%%%%%%%%%%%%%%%%
%%%%%%%%%%%%%%%%%%%%%%%%%%%%%%%%%%%%

\subsection{A $n\times n$ Meixner type} 
\label{sec:nxn-example-meixner}

The representation $V$ has countable dimension, and we consider the $M_n(\dC)$-bases of $V^n$ given by
\begin{equation}
    e_k=(\ket{nk},\ldots,\ket{nk+n-1}), \qquad \Phi_j=(\ket{\phi_{nj}}, \ldots, \ket{\phi_{nj+n-1}}), \qquad  k,j=0,1\ldots
\end{equation}
The actions of the operators $P^n$ and $H^n$ on the $M_n(\dC)$-bases $(e_k),$ and $(\Phi_j)$ are described in Proposition \ref{prop:action-Pn-basis}. 
The matrix entries of $(e_0,\Phi_j)$ are given by 
\begin{equation}
    (e_0,\Phi_j)_{i,\ell}=\bra{i-1}\ket{\phi_{nj+\ell-1}} = M_{i-1}(nj+\ell-1;c,\beta)h_{i-1}^{-\frac{1}{2}}w_{nj+\ell-1}^\frac12 \qquad i,\ell=1,\ldots, n. 
\end{equation}
The invertibility of $(e_0,\Phi_j)$ is proved as in the $\mathfrak{su}(2)$ case using the $n\times n$ matrix $L(j)$ defined as
\begin{equation}
\label{eq:alternant-meixner}
L(j)_{i,\ell}=h_{i-1}^{-\frac{1}{2}}M_{i-1}(nj+\ell-1;c,\beta), \qquad i,\ell=1,\ldots, n.
\end{equation}
from which we can express
\begin{equation}
(e_0,\Phi_j)=L(j)\diag(w_{nj}^\frac12,\ldots,w_{nj+n-1}^\frac12).
\end{equation} 
Since the matrices $A_k$ constructed using Proposition \ref{prop:action-Pn-basis} are invertible, we can proceed as outlined in Section \ref{sec:algebraic-interpretation-mvop} and obtain the family of MVOPs $\Pi_k(x)$ that satisfy the recurrence relation 
 \begin{equation}
\Pi_k(\theta_j)\theta_j=A_{k+1}^*\Pi_{k+1}(\theta_j)+B_k\Pi_k(\theta_j)+A_k\Pi_{k-1}(\theta_j), \qquad 
\Pi_{-1}(\theta_j)=0,\qquad \Pi_0(\theta_j)=I_n,
 \end{equation}
 where $\theta_j=( e_0, \Phi_j) \alpha_j^*( e_0, \Phi_j)^{-1}.$ The family $\Pi_k(x)$ satisfy the orthogonality relation
\begin{equation}
\label{eq:orth-meixner-nxn}
\sum_{j=0}^\infty \Pi_{k}(\theta_j)W(j)\Pi_{\ell}(\theta_j)^\ast=\delta_{k,\ell},\qquad W(j)=L(j)D(j)L(j)^\ast,
\end{equation} 
where the matrix $L(j)$ is defined in \eqref{eq:alternant-krawtchouk} and $D(j)=\diag(w_{nj}, \ldots, w_{nj+n-1}).$ Moreover, the polynomials satisfy the difference equation
\begin{multline}
   \Lambda_k\Pi_k(\theta_j)=\Pi_k(\theta_{j+1})(e_0,\Phi_{j+1})A_{j+1}(e_0,\Phi_j)+\Pi_k(\theta_{j})(e_0,\Phi_{j})B_{j}(e_0,\Phi_j)\\+\Pi_k(\theta_{j-1})(e_0,\Phi_{j-1})A_{j}^\ast(e_0,\Phi_j). 
\end{multline}
We conclude the section by presenting the explicit expressions in the $2\times 2$ case. The alternant matrix $L$ defined in \eqref{eq:alternant-general} is 
\begin{equation}
    L(j)=\begin{pmatrix}
h_{0}^{-\frac{1}{2}}&h_{0}^{-\frac{1}{2}}\\
h_{1}^{-\frac{1}{2}}M_1(2j;c,\beta)&h_{1}^{-\frac{1}{2}}M_1(2j+1;c,\beta)
\end{pmatrix},
\end{equation}
and its inverse is 
\begin{equation}
    L(j)^{-1}=\frac{c\beta}{c-1}\begin{pmatrix}
h_{0}^{\frac{1}{2}}M_1(2j+1;c,\beta)&-h_{1}^{\frac{1}{2}}\\
-h_{0}^{\frac{1}{2}}M_1(2j;c,\beta)&h_{1}^{\frac{1}{2}}
\end{pmatrix}.
\end{equation}
The weight function is $W(j)=L(j)\diag(w_{2j},w_{2j+1})L(j)^\ast.$ The matrix $A_k$ on the recurrence relation and difference equation is given by
\begin{equation}
    A_{k}=\begin{pmatrix}
\frac{\sinh(a)^2}{4}\rho^-_{2 k} \rho^-_{2 k-1} & -\frac{\sinh(a)\cosh(a)}{2}\rho^-_{2 k}(\lambda_{2 k}+\lambda_{2 k -1})  
\\
0 & \frac{\sinh(a)^2}{4}\rho^-_{2 k}\rho^-_{2 k +1}
\end{pmatrix},
\end{equation}
and the matrix $B_k$ is given by
\begin{equation}
    B_k=\begin{pmatrix}
\frac{\sinh(a)^2}{4}( \rho_{2 k}^{-2} + \rho _{2 k+1}^{-2}) +\cosh(a)^2 \lambda_{2 k}^{2}& -\frac{\sinh(a)\cosh(a)}{2}(\lambda_{2 k} \rho_{2 k+1}^- + \rho_{2 k+1}^-\lambda_{2 k +1})
\\
- \frac{\sinh(a)\cosh(a)}{2}(\lambda_{2 k} \rho_{2 k+1}^- + \rho_{2 k+1}^-\lambda_{2 k +1}) & \frac{\sinh(a)^2}{4}( \rho_{2 k +1}^{-2}+\rho_{2 k+2}^{-2})+\cosh(a)^{2} \lambda_{2 k +1}^{2}
\end{pmatrix},
\end{equation}
where $
\lambda_k=\frac{\beta}{2}+k$ and $ \rho^-_k=\sqrt{k(\beta+k-1)}$.
Finally, the MVOPs are given by $\Pi_k(\theta_j)=R_k(j)L(j)^{-1},$ where
\begin{equation}
    R_k(j)=\begin{pmatrix}
M_{2k}(2j;c,\nu)h_{2k}^{-\frac{1}{2}}&M_{2k}(2j+1;c,\nu)h_{2k}^{-\frac{1}{2}}\\
M_{2k+1}(2j;c,\nu)h_{2k+1}^{-\frac{1}{2}}&M_{2k+1}(2j+1;c,\nu)h_{2k+1}^{-\frac{1}{2}}
\end{pmatrix},
\end{equation}
where $h_k$ is defined in \eqref{eq:weight-norm-scalar-meixner}.

\begin{remark}
In both the Krawtchouk and Meixner analogues, extensive symbolic computations indicate that the weight matrices $W$ are
irreducible, see \cite{TZreducibility} and \cite{KRreducibility}. A complete proof of this fact for the $2\times2$ case can be carried out applying \cite[Theorem 2.3]{TZreducibility}.
\end{remark}
%%%%%%%%%%%%%%%%%%%%%%%%%%%%%%%%%
%%%%%%%%%%%%%%%%%%%%%%%%%%%%%%%%%
%%%%%%%%%%%%%%%%%%%%%%%%%%%%%%%%%
%%%%%%%%%%%%%%%%%%%%%%%%%%%%%%%%%
\section{Outlook}
In this article we were able to develop a general method to obtain matrix analogs of scalar valued orthogonal polynomials and to apply it successfully on various examples. An interesting feature is that for the same algebra, here $\mathfrak{su}(2)$, we were able to produce two non equivalent examples. In this work, we provided an explicit formula for $\Pi_k(\theta_j)$ as a function in $j$. A future development would be to produce a formula expanding our polynomials $\Pi_k$ in the $X^k$ basis. This raises the question of the classification of MVOPs. In the case of finite families of MVOPs satisfying a difference equation one approach could be a generalization of Leonard pairs in the context of $M_n(\dC)$-modules as developed in Section $2$. It would also be quite interesting to find applications of these MVOPs in models of interest in physics.

\subsection*{Conflict of interest:} On behalf of all authors, the corresponding author states that there is no conflict of interest.

\subsection*{Use of generative AI:} On behalf of all authors, the corresponding author states that generative AI tools were not used for this article.

\subsection*{Data availability statement:} The authors declare that the data supporting the findings of this study are available within the paper.

\bibliographystyle{plain} 

\bibliography{Bibliography} 
\end{document}